\documentclass[12pt,reqno]{article}

\usepackage{amsfonts}
\usepackage{amsmath}
\usepackage{amsbsy}
\usepackage{graphicx}
\usepackage{amssymb,latexsym}
\usepackage{graphicx}
\usepackage{caption}
\usepackage{subcaption}
\numberwithin{equation}{section}
\newtheorem{theorem}{Theorem}

\newtheorem{corollary}[theorem]{Corollary}

\newenvironment{proof}[1][Proof]{\noindent\textbf{#1} }{\ \rule{0.5em}{0.5em}}

\begin{document}

\title{On the Relation Between $G_{2}^*$ Structures and Almost Paracontact Structures}
\begin{titlepage}
\author{
\c{S}irin AKTAY\footnote{E.mail:
sirins@eskisehir.edu.tr, orcid id: https://orcid.org/0000-000302792-3481
} \\
{\small Department of Mathematics, Eski\c{s}ehir Technical University, 26470 Eski\c{s}ehir, Turkey}}

\date{ }

\maketitle

\bigskip

\begin{abstract}
\noindent \noindent We investigate 7-dimensional almost para-contact metric structures induced by the 3-forms of $G_{2}^*$ structures. We calculate the projections that determine to which class the almost para-contact structure belongs, by using the properties of the $G_{2}^*$ structures.



\noindent
\end{abstract}
\end{titlepage}

\section{Introduction}

Manifolds with almost paracontact structures were first defined by Kaneyuki and Williams in \cite{kaneyuki}. Zamkovoy provided all the technical apparatus needed in \cite{zamkovoycanonical}. After these remarkable works, almost paracontact metric manifolds were written as a direct sum of $12$ subspaces  with respect to the symmetry properties of the Levi-Civita covariant derivative of the fundamental 2-form in \cite{nakova, zamkovoyrevisited}.

Almost paracontact metric structures induced by manifolds with $G_{2}^*$ structures were constructed in \cite{bizimparacontact} and existence of some classes were investigated. Our aim in this study is to get further results by calculating projections given in \cite{zamkovoyrevisited} on each of twelve subspaces of almost paracontact metric structures. Also, we provide some examples.

\section{Preliminaries}

Consider $\mathbb{R}^7$ with the metric $g_{4,3}$ having the signature $(-,-,-,-,+,+,+)$. The group $G_{2}^*$ is defined as
$$G_{2}^*=\{ g \in GL(7,\mathbb{R}) \ | \ g^*\varphi=\varphi \},$$
where
$$\varphi=-e^{127}-e^{135}+e^{146}+e^{236}+e^{245}-e^{347}+e^{567}.$$

The basis $e^1, \ldots, e^7$ is the metric dual of the standard basis $\{e_1,\ldots,e_7\}$. If the structure group of a 7-dimensional oriented manifold $M$ reduces to the group $G_{2}^*$, then $M$ is called a manifold with $G_{2}^*$ structure. Then, $M$ has a 3-form $\varphi$ such that
for all $p\in M$, the space $(T_pM,\varphi_p)$ is isomorphic to $(\mathbb{R}^7,\varphi)$. The 3-form $\varphi$ is said to be the $G_{2}^*$ structure or the fundamental 3-form of $M$. $\varphi$ induces a metric $g_{4,3}$ with signature $(-,-,-,-,+,+,+)$, a volume form $d_{vol}$, a cross product $P$ on $M$ by the following equations:
\begin{equation}\label{ucform}\varphi(X,Y,Z)=g_{4,3}(P(X,Y),Z),\end{equation}
\begin{equation}\label{varphi}(X\lrcorner\varphi)\wedge(Y\lrcorner\varphi)\wedge\varphi=6g_{4,3}(X,Y)d_{vol},\end{equation}
for all vector fields $X, Y, Z$ \cite{bryant, kath}.

If an odd dimensional smooth manifold $M$ has an endomorphism $\phi$, a vector field $\xi$ and a 1-form $\eta$ satisfying
$$\phi^2=I-\eta\otimes\xi, \ \ \eta(\xi)=1,$$
and if there is a distribution $$\mathbb{D}: p\in M\longrightarrow \mathbb{D}_p\subset T_pM$$
such that $\mathbb{D}_p=Ker \eta$, then $M$ is called a manifold with an almost paracontact structure $(\phi,\xi,\eta)$. If $M$ also has a semi-Riemannian metric $g$ with the property that
$$g(\phi X,\phi Y)=-g(X,Y)+\eta(X)\eta(Y),$$
the quadruple $(\phi,\xi,\eta,g)$ is called an almost paracontact metric structure on $M$ with compatible metric $g$.

The 2-form defined by $\Phi(X,Y)=g(\phi X,Y)$ is called the fundamental 2-form of the almost paracontact metric structure.

There are $2^{12}$ classes of almost paracontact metric structures. The space $\mathcal{F}$ the Levi-Civita covariant derivative of the fundamental 2-form belongs to was decomposed into subspaces $W_i$, $i=1,2,3,4$ and then to twelve subspaces $\mathbb{G}_i$, $i=1,\ldots,12$. Then $\nabla\Phi$ can be represented uniquely in the form $\nabla\Phi=F^{W_1}+F^{W_2}+F^{W_3}+F^{W_4}$ and $\nabla\Phi=F^1+F^2+\ldots+F^{12}$, where $F^{W_i}\in W_i$ and $F^{i}\in \mathbb{G}_i$.

An almost paracontact metric manifold $M$ with structure $(\phi,\xi,\eta,g)$ is in the class $\mathbb{G}_i\oplus\mathbb{G}_j\oplus\ldots$ if and only if $\nabla\Phi$ is the sum of corresponding projections $F^{i}+F^{j}+\ldots$. For further information on the classification and projections, see \cite{nakova, zamkovoyrevisited}.

Let $M$ be a manifold with $G_2^*$ structure $\varphi$ with the metric $g_{4,3}$ and the cross product P. Choose a vector field $\xi$ satisfying $g_{4,3}(\xi,\xi)=-1$. Then for $\phi(X)=P(\xi,X)$, $g=-g_{4,3}$ and $\eta(X)=g(\xi, X)$, the quadruple $(\phi,\xi,\eta,g)$ is an almost paracontact metric structure on $M$ induced by the $G_2^*$ structure $\varphi$. For the Levi-Civita covariant derivatives of metrics, we have $\nabla^g=\nabla^{g_{4,3}}$ and we denote this derivatives by $\nabla$. We have the relation
\begin{equation}\label{covariant}(\nabla_X\Phi)(Y,Z)=-(\nabla_X\varphi)(\xi,Y,Z)-\varphi(\nabla_X\xi,Y,Z),\end{equation}
see \cite{bizimparacontact}. For almost contact structures induced by manifolds with $G_2$ structures, refer to \cite{bizimcontact}.

\section{Projections identifying almost paracontact structures}

In this section we assume that $M$ is a manifold with $G_2^*$ structure $\varphi$, the metric $g_{4,3}$, the
cross product $P$ and $\xi$ is a vector field with the property that $g_{4,3}(\xi,\xi)=-1$. We denote the almost paracontact metric structure obtained from the $G_2^*$ structure $\varphi$ by $(\phi,\xi,\eta,g)$. We calculate the projections $F^{i}$ and obtain conditions on $\xi$ such that the corresponding almost paracontact structure has a summand from a certain class. Vector fields are denoted by capital letters, such as $X$, $Y$, $Z$.

\begin{theorem}\label{thm1} Let $\varphi$ be any $G_{2}^*$ structure and $(\phi, \xi, \eta, g)$ an almost para-contact metric structure induced by $\varphi$. Then $\mathcal{F}^{12}= 0$ if and only if $\nabla_{\xi}\xi= 0$.
\end{theorem}
\begin{proof}
Since
\begin{eqnarray*}
F(\xi,\xi,\phi^2Z)&=&(\nabla_{\xi}\Phi)(\xi,\phi^2Z)\\
&=&-(\nabla_{\xi}\varphi)(\xi,\xi,\phi^2Z)-\varphi(\nabla_{\xi}\xi,\xi,\phi^2Z)\\
&=&-\varphi(\nabla_{\xi}\xi,\xi,\phi^2Z)\\
&=&-\varphi(\nabla_{\xi}\xi,\xi,Z),
\end{eqnarray*}
we get
\begin{equation}\label{F12}\mathcal{F}^{12}(X,Y,Z)=\eta(X)\{-\eta(Y)\varphi(\nabla_{\xi}\xi,\xi,Z)+\eta(Z)\varphi(\nabla_{\xi}\xi,\xi,Y)\}.\end{equation}

Clearly if $\nabla_{\xi}\xi=0$, then $\mathcal{F}^{12}=0$. That is the a.p.m.s. does not contain a summand from the class $\mathbb{G}_{12}$.

Now replace $Z$ by $\xi$ and $Y$ by $\phi(Y)$ in (\ref{F12}). Then
$$\begin{array}{rcl}
\mathcal{F}^{12}(X,\phi Y,\xi)&=&\eta(X)\{\varphi(\nabla_{\xi}\xi,\xi,\phi Y)\}\\
&=&\eta(X)\{\varphi(\xi,\phi Y,\nabla_{\xi}\xi)\}\\
&=&\eta(X)\{g_{4,3}(P(\xi,P(\xi,Y)), \nabla_{\xi}\xi)\}\\
&=&\eta(X)\{g_{4,3}(Y,\nabla_{\xi}\xi)\}
\end{array}
$$

If $\nabla_{\xi}\xi\neq 0$, then $g_{4,3}(Y,\nabla_{\xi}\xi)\neq 0$ since the metric $g_{4,3}$ is non-degenerate. Thus $\mathcal{F}^{12}(X,\phi \nabla_{\xi}\xi,\xi)=\eta(X)\{g_{4,3}(\nabla_{\xi}\xi,\nabla_{\xi}\xi)\}\neq 0$, or equivalently, if $\mathcal{F}^{12}=0$, then $\nabla_{\xi}\xi=0$.
\end{proof}

Now we calculate the projection $\mathcal{F}^{11}$:

\begin{theorem}\label{thm2}If the $G_{2}^*$ structure $\varphi$ has the property that $\xi\lrcorner\nabla_{\xi}\varphi=0$ and $\xi$ satisfies $\nabla_{\xi}\xi=0$, then $\mathcal{F}^{11}=0$.
\end{theorem}

\begin{proof}
$$
\begin{array}{rcl}
\mathcal{F}^{11}(X,Y,Z)&=&\eta(X)F(\xi,\phi^2Y,\phi^2Z)\\
&=&\eta(X)\{-(\nabla_{\xi}\varphi)(\xi,\phi^2Y,\phi^2Z)-\varphi(\nabla_{\xi}\xi,\phi^2Y,\phi^2Z)\}\\
&=&\eta(X)\left\{-(\nabla_{\xi}\varphi)(\xi,Y,Z)-\varphi(\nabla_{\xi}\xi,Y,Z)\right.\\
&&+\left.\eta(Z)\varphi(\nabla_{\xi}\xi,Y,\xi)+\eta(Y)\varphi(\nabla_{\xi}\xi,\xi,Z)\right\}
\end{array}
$$
and clearly if $\xi\lrcorner\nabla_{\xi}\varphi=0$ and $\nabla_{\xi}\xi=0$, then $\mathcal{F}^{11}=0$.
\end{proof}

In this case, the structure does not contain a part from $\mathbb{G}_{11}$. Note that nearly-parallel $G_{2}^*$ structures satisfy $\xi\lrcorner\nabla_{\xi}\varphi=0$.

More generally:

\begin{theorem} Let $\xi\lrcorner\nabla_{\xi}\varphi=0$. If there is a vector field $\xi$ satisfying $g_{4,3}(\xi,\xi)=-1$ and
$$-P(\nabla_{\xi}\xi,Y)-\varphi(\nabla_{\xi}\xi,Y,\xi)\xi+P(\nabla_{\xi}\xi,\xi)=0,$$
then $\mathcal{F}^{11}=0$. Otherwise, $\mathcal{F}^{11}\neq 0$.
\end{theorem}

\begin{proof}
$\mathcal{F}^{11}=0$ if and only if
$$
\begin{array}{rcl}
0&=&-(\nabla_{\xi}\varphi)(\xi,Y,Z)-\varphi(\nabla_{\xi}\xi,Y,Z)+\eta(Z)\varphi(\nabla_{\xi}\xi,Y,\xi)+\eta(Y)\varphi(\nabla_{\xi}\xi,\xi,Z)\\
&&=-g_{4,3}(P(\nabla_{\xi}\xi,Y),Z)-g_{4,3}(\varphi(\nabla_{\xi}\xi,Y,\xi)\xi,Z)+g_{4,3}(\eta(Y)P(\nabla_{\xi}\xi,\xi),Z)
\end{array}$$
Since the metric $g_{4,3}$ is non-degenarate, the result follows.
\end{proof}

It is known that $W_2=\mathbb{G}_5\oplus\mathbb{G}_6\oplus\ldots\mathbb{G}_{10}$ and $F^{W_2}=F^5\oplus F^6\oplus\ldots\oplus F^{10}$, \cite{zamkovoyrevisited}. We state the theorem below.

\begin{theorem} For an arbitrary $G_2^*$ structure $\varphi$ and
an a.p.m.s $(\phi, \xi, \eta, g)$ induced by $\varphi$, if $\xi$ is parallel, then $\mathcal{F}^{W_2}=0$.
\end{theorem}
\begin{proof} The identity (\ref{covariant}) implies
$$
\begin{array}{rcl}
F(\phi^2X,\phi^2Z,\xi)&=&-\varphi(\nabla_{\phi^2X}\xi,\phi^2Z,\xi)\\
&=&-\varphi(\nabla_{X-\eta(X)\xi}\xi,Z,\xi)\\
&=&-\varphi(\nabla_{X}\xi,Z,\xi)+\eta(X)\varphi(\nabla_{\xi}\xi,Z,\xi).
\end{array}$$
Thus
$$
\begin{array}{rcl}
\mathcal{F}^{W_2}(X,Y,Z)&=&-\eta(Y)F(\phi^2X,\phi^2Z,\xi)+\eta(Z)F(\phi^2X,\phi^2Y,\xi)\\
&=&\eta(Y)\varphi(\nabla_{X}\xi,Z,\xi)-\eta(Y)\eta(X)\varphi(\nabla_{\xi}\xi,Z,\xi)\\
&&-\eta(Z)\varphi(\nabla_{X}\xi,Y,\xi)+\eta(Z)\eta(X)\varphi(\nabla_{\xi}\xi,Y,\xi),
\end{array}
$$
which is zero if $\xi$ is parallel.
\end{proof}

If $\xi$ is parallel, $F^{W_2}=0$ and thus $F^5=F^6=\ldots=F^{10}=0$. Also by Theorem \ref{thm1}, $F^{12}$=0. So if the characteristic vector field is parallel, the a.p.m.s. is in $\mathbb{G}_1\oplus\mathbb{G}_2\oplus\mathbb{G}_{3}\oplus\mathbb{G}_{4}\oplus\mathbb{G}_{11}$.

In addition, we have

 \begin{theorem} \label{thm5}
 If $\mathcal{F}^{W_2}=0$, then $\nabla_{\phi X}\xi = 0$ for an almost paracont metric structure obtained from a general $G_2^*$ structure.
 \end{theorem}
\begin{proof} Assume that
$$
\begin{array}{rcl}
\mathcal{F}^{W_2}(X,Y,Z)
&=&\eta(Y)\varphi(\nabla_{X}\xi,Z,\xi)-\eta(Y)\eta(X)\varphi(\nabla_{\xi}\xi,Z,\xi)\\
&&-\eta(Z)\varphi(\nabla_{X}\xi,Y,\xi)+\eta(Z)\eta(X)\varphi(\nabla_{\xi}\xi,Y,\xi)=0.
\end{array}
$$
Then
$$-\eta(Y)P(\nabla_X\xi,\xi)+\eta(Y)\eta(X)P(\nabla_{\xi}\xi,\xi)+\{\varphi(\nabla_X\xi,Y,\xi)-\eta(X)\varphi(\nabla_{\xi}\xi,Y,\xi)\}\xi=0.$$
Take the inner product $\xi$ which gives
$$\eta(X)\varphi(\nabla_{\xi}\xi,Y,\xi)-\varphi(\nabla_X\xi,Y,\xi)=0.$$
Replacing $X$ by $\phi(X)$ yields
$$\varphi(\nabla_{\phi X}\xi,Y,\xi)=-g_{4,3}(P(\nabla_{\phi X}\xi,\xi),Y)=0.$$
Since $g_{4,3}$ is non-degenerate, we obtain $P(\nabla_{\phi X}\xi,\xi)=0$, which holds if and only if $\nabla_{\phi X}\xi$ and $\xi$ are linearly dependent. Let $\nabla_{\phi X}\xi=k\xi$ for a function $k$. Then
$$0=g_{4,3}(\nabla_{\phi X}\xi,\xi)=kg_{4,3}(\xi,\xi)=-k,$$
and thus $\nabla_{\phi X}\xi=k\xi=0$ for all $X$.
\end{proof}

Next we give a condition for the a.p.m.s to be normal.

\begin{theorem}\label{thmnormal} For any $G_2^*$ structure, $(\phi, \xi, \eta, g)$ is normal if and only if
\begin{equation}\label{normal}
\nabla_XZ-X[\eta(Z)]-P(\xi,\nabla_X(P(\xi,Z)))+\nabla_{P(\xi,X)}(P(\xi,Z))-P(\xi,\nabla_{P(\xi,X)}Z)=0.
\end{equation}

\end{theorem}

\begin{proof}The normality condition is
\begin{equation}\label{normalitycondition}F(X,Y,\phi Z)+F(\phi X,Y,Z)+F(X,\phi Y,\eta(Z)\xi)=0,\end{equation}
see \cite{zamkovoyrevisited}. We write this equation by using the properties of the $G_2^*$ structure $\varphi$.

$$
\begin{array}{rcl}
0&=&F(X,Y,\phi Z)+F(\phi X,Y,Z)+F(X,\phi Y,\eta(Z)\xi)\\
&=&-(\nabla_X\varphi)(\xi,Y,\phi Z)-\varphi(\nabla_X\xi,Y,\phi Z)-(\nabla_{\phi X}\varphi)(\xi,Y,Z)\\
&&-\varphi(\nabla_{\phi X}\xi,Y,Z)-(\nabla_X\varphi)(\xi,\phi Y,\eta(Z)\xi)-\varphi(\nabla_X\xi,\phi Y,\eta(Z)\xi)\\
&=&-X[\varphi(\xi,Y,\phi Z)]+\varphi(\xi,\nabla_XY,\phi Z)+\varphi(\xi,Y,\nabla_X(\phi Z))\\
&&-\phi X[\varphi(\xi,Y,Z)]+\varphi(\xi,\nabla_{\phi X}Y,Z)+\varphi(\xi,Y,\nabla_{\phi X}Z)-\varphi(\nabla_X\xi,\phi Y,\eta(Z)\xi)\\
&=&g_{4,3}(\nabla_XZ,Y)+g_{4,3}(\xi,Y)g_{4,3}(\nabla_X\xi,Z)+g_{4,3}(\xi,Y)g_{4,3}(\xi,\nabla_XZ)\\
&&+\varphi(\xi,Y,\nabla_X(\phi Z))+g_{4,3}(\nabla_{\phi X}(\phi Z),Y)+\varphi(\xi,Y,\nabla_{\phi X}Z).
\end{array}
$$
The result follows since $g_{4,3}$ is non-degenerate.
\end{proof}

By Theorem \ref{thmnormal}, we state the following.

\begin{corollary}\label{normallik} If $(\phi, \xi, \eta, g)$ is a normal structure obtained from a general $G_2^*$ structure, then $\nabla_{\xi}\xi= 0$.
\end{corollary}
\begin{proof} If $(\phi, \xi, \eta, g)$ is normal, then (\ref{normal}) holds. In (\ref{normal}), take the inner product of both sides with $\xi$ with respect to the metric $g_{4,3}$.
Then we have
$$-g_{4,3}(\nabla_X\xi,Z)+g_{4,3}(\nabla_{\phi X}(\phi Z),\xi)=0.$$
Replacing $X$ by $\xi$, we get $g_{4,3}(\nabla_{\xi}\xi,Z)=0$, which implies $\nabla_{\xi}\xi=0$.
\end{proof}

Now we study the property of the characteristic vector field of a paracontact structure.

\begin{theorem}\label{paracontact} For a paracontact $(\phi, \xi, \eta, g)$ obtained from any $G_2^*$, we have $\nabla_{\xi}\xi= 0$.
\end{theorem}
\begin{proof} Let $(\phi, \xi, \eta, g)$ be paracontact. By Proposition 3.2 in \cite{bizimparacontact},
$$2g_{4,3}(P(\xi,X),Y)=g_{4,3}(\nabla_X\xi,Y)-g_{4,3}(\nabla_Y\xi,X).$$

For $Y=\xi$, we get $g_{4,3}(\nabla_{\xi}\xi,X)=0$, and thus $\nabla_{\xi}\xi=0$.
\end{proof}

Now we consider an a.p.m.s with Killing characteristic vector field obtained from a $G_2^*$ structure, that is $g(\nabla_X\xi,Y)=-g(\nabla_Y\xi,X)$. By Proposition 4.7 in \cite{zamkovoyrevisited}, the a.p.m.s can only be in the classes $\mathbb{G}_1$, $\mathbb{G}_2$, $\mathbb{G}_3$, $\mathbb{G}_4$, $\mathbb{G}_5$, $\mathbb{G}_8$, $\mathbb{G}_9$, $\mathbb{G}_{11}$ and in their direct sums. That is, $F^5=F^6=F^7=F^{12}=0$ and
$$F=F^1+F^2+F^3+F^4+F^8+F^9+F^{11}.$$

By using the properties of $G_2^*$ structures, we state the followings.

\begin{theorem} Assume that $(\phi,\xi,\eta,g)$ is induced by a general $\varphi$ and $\xi$ is Killing. Then $F^{11}=0$ if and only if the $G_2^*$ structure satisfies $\xi\lrcorner\nabla_{\xi}\varphi=0$.
\end{theorem}
\begin{proof} Assume that $\xi\lrcorner\nabla_{\xi}\varphi=0$. Since $\xi$ is Killing, $\nabla_{\xi}\xi=0$ and the result follows by Theorem \ref{thm2}.

Conversely if $F^{11}=0$,
$$\begin{array}{rcl}
\mathcal{F}^{11}(X,Y,Z)
&=&\eta(X)\left\{-(\nabla_{\xi}\varphi)(\xi,Y,Z)-\varphi(\nabla_{\xi}\xi,Y,Z)\right.\\
&&\left.+\eta(Z)\varphi(\nabla_{\xi}\xi,Y,\xi)+\eta(Y)\varphi(\nabla_{\xi}\xi,\xi,Z)\right\}\\
&=&\eta(X)\{-(\nabla_{\xi}\varphi)(\xi,Y,Z)\}\\
&=&0
\end{array}$$
and thus $\xi\lrcorner\nabla_{\xi}\varphi=0$.
\end{proof}

Next we calculate $F^9$ for $\xi$ Killing.
\begin{theorem} Given $(\phi,\xi,\eta,g)$ obtained by $\varphi$ with $\xi$ Killing, $F^{9}=0$ if and only if
$$\eta(Y)(P(\nabla_X\xi,\xi)+\nabla_{\phi X}\xi)+(g(\nabla_{\phi Y}\xi,X)+g(\nabla_{\phi X}\xi,Y))\xi=0.$$
\end{theorem}

\begin{proof}Since $\xi$ is Killing,
$$g(\nabla_{\phi X}\xi,Z)=-g(\nabla_Z\xi,\phi X)=g_{4,3}(\nabla_Z\xi,\phi X)=-\varphi(\nabla_Z\xi,X,\xi)$$
and $$g(\nabla_{\phi Z}\xi,X)=-\varphi(\nabla_X\xi,Z,\xi),$$
$$g(\nabla_{\phi X}\xi,Y)=-\varphi(\nabla_Y\xi,X,\xi),$$
$$g(\nabla_{\phi Y}\xi,X)=-\varphi(\nabla_X\xi,Y,\xi).$$ By direct calculation,
$$
\begin{array}{rcl}
4\mathcal{F}^{9}(X,Y,Z)&=&-\eta(Y)\left\{-\varphi(\nabla_X\xi,Z,\xi)\right.\\
&&\left.-g(\nabla_{\phi X}\xi,Z)+\varphi(\nabla_Z\xi,X,\xi)+g(\nabla_{\phi Z}\xi,X)\right\}\\
&&+\eta(Z)\left\{-\varphi(\nabla_X\xi,Y,\xi)\right.\\
&&\left.-g(\nabla_{\phi X}\xi,Y)+\varphi(\nabla_Y\xi,X,\xi)+g(\nabla_{\phi Y}\xi,X)\right\}.\\
&=&-\eta(Y)\left\{-2\varphi(\nabla_X\xi,Z,\xi)-2g(\nabla_{\phi X}\xi,Z)\right\}\\
&&+\eta(Z)\left\{2g(\nabla_{\phi Y}\xi,X)-2g(\nabla_{\phi X}\xi,Y)\right\}\\
&=&-\eta(Y)\left\{-2g(P(\nabla_X\xi,\xi),Z)-2g(\nabla_{\phi X}\xi,Z)\right\}\\
&&+g(\xi,Z)\left\{2g(\nabla_{\phi Y}\xi,X)-2g(\nabla_{\phi X}\xi,Y)\right\}.
\end{array}
$$
Thus
$$2\mathcal{F}^{9}(X,Y,Z)=g(\eta(Y)(P(\nabla_X\xi,\xi)+\nabla_{\phi X}\xi)+(g(\nabla_{\phi Y}\xi,X)+g(\nabla_{\phi X}\xi,Y))\xi,Z)$$
and $F^9=0$ if and only if $$\eta(Y)(P(\nabla_X\xi,\xi)+\nabla_{\phi X}\xi)+(g(\nabla_{\phi Y}\xi,X)+g(\nabla_{\phi X}\xi,Y))\xi=0.$$
\end{proof}

Next we calculate $F^8$.

\begin{theorem} Let $(\phi,\xi,\eta,g)$ be induced by $\varphi$ and let $\xi$ be Killing. Then $F^{8}=0$ if and only if
$$\eta(Y)(P(\nabla_X\xi,\xi)-\nabla_{\phi X}\xi)-(\varphi(\nabla_X\xi,Y,\xi)+\varphi(\nabla_Y\xi,X,\xi))\xi=0.$$
\end{theorem}

\begin{proof}If $\xi$ is Killing, then
$$F^5(X,Y,Z)=\frac{\theta_{F^5}(\xi)}{2n}\{\eta(Y)g(\phi X,\phi Z)-\eta(Z)g(\phi X,\phi Y)\}=0,$$
see \cite{zamkovoyrevisited}, and
$$g(\nabla_{\phi X}\xi,Z)=-g(\nabla_Z\xi,\phi X)=g_{4,3}(\nabla_Z\xi,\phi X)=-\varphi(\nabla_Z\xi,X,\xi),$$

$$g(\nabla_{\phi Z}\xi,X)=-\varphi(\nabla_X\xi,Z,\xi),$$

$$g(\nabla_{\phi X}\xi,Y)=-\varphi(\nabla_Y\xi,X,\xi), \ \ g(\nabla_{\phi Y}\xi,X)=-\varphi(\nabla_X\xi,Y,\xi).$$

Thus
$$
\begin{array}{rcl}
\mathcal{F}^{8}(X,Y,Z)&=&-\frac{1}{4}\eta(Y)\left\{F(\phi^2X,\phi^2Z,\xi)-F(\phi X,\phi Z,\xi)\right.\\
&&\left.+F(\phi^2Z,\phi^2X,\xi)-F(\phi Z,\phi X,\xi)\right\}\\
&&+\frac{1}{4}\eta(Z)\left\{F(\phi^2X,\phi^2Y,\xi)-F(\phi X,\phi Y,\xi)\right.\\
&&\left.+F(\phi^2Y,\phi^2X,\xi)-F(\phi Y,\phi X,\xi)\right\}\\
&=&\frac{1}{2}\eta(Y)\left\{\varphi(\nabla_X\xi,Z,\xi)+\varphi(\nabla_Z\xi,X,\xi)\right\}\\
&&-\frac{1}{2}\eta(Z)\left\{\varphi(\nabla_X\xi,Y,\xi)+\varphi(\nabla_Y\xi,X,\xi)\right\}\\
&=&\frac{1}{2}\eta(Y)\left\{g(P(\nabla_X\xi,\xi),Z)-g(\nabla_{\phi X}\xi,Z)\right\}\\
&&-\frac{1}{2}g(\xi,Z)\left\{\varphi(\nabla_X\xi,Y,\xi)+\varphi(\nabla_Y\xi,X,\xi)\right\}.
\end{array}
$$
Therefore
$$2\mathcal{F}^{8}(X,Y,Z)=g(\eta(Y)(P(\nabla_X\xi,\xi)-\nabla_{\phi X}\xi)-\varphi(\nabla_X\xi,Y,\xi)-\varphi(\nabla_Y\xi,X,\xi))\xi,Z)=0$$
if and only if $$\eta(Y)(P(\nabla_X\xi,\xi)-\nabla_{\phi X}\xi)-(\varphi(\nabla_X\xi,Y,\xi)+\varphi(\nabla_Y\xi,X,\xi))\xi=0.$$
\end{proof}

It is not easy to obtain results on $F^1$, $F^2$, $F^3$ and $F^4$ for an a.p.m.s obtained from an arbitrary $G_2^*$ structure. For example,
$$F^3(X,Y,Z)+F^4(X,Y,Z)=\frac{1}{2}\{F(\phi^2X,\phi^2Y,\phi^2Z)+F(\phi X,\phi^2Y,\phi Z)\}.$$
If $\xi$ is parallel, then
$$\begin{array}{rcl}
F(\phi^2X,\phi^2Y,\phi^2Z)+F(\phi X,\phi^2Y,\phi Z)&=&-(\nabla_X\varphi)(\xi,Y,Z)\\
&&+\eta(X)(\nabla_{\xi}\varphi)(\xi,Y,Z)-(\nabla_{\phi(X)}\varphi)(\xi,Y,Z)
\end{array}$$
and $F^3+F^4$ need not be zero. Now we give some examples.

\section{Examples}

Consider the 7-dimensional Lie algebra $\mathcal{L}$ with basis $f_1, f_2, \ldots, f_7$ and nonzero brackets
$$[f_3,f_7]=f_1, [f_4,f_7]=f_3, [f_5,f_7]=f_2, [f_6,f_7]=f_5.$$
The 3-form

\begin{equation}\label{3form} \varphi=-f^{156}-f^{236}+f^{245}-\frac{1}{2}f^{127}-f^{347},
\end{equation}
where $f^{ijk}$ denotes $f^i\wedge f^j\wedge f^k$ is a $G_2^*$ structure inducing
$$g_{\varphi}=-2f^2.f^2+f^1.f^7+2f^3.f^6-2f^4.f^5,$$
see \cite{freiberg}. Compare the coefficient of $f^2.f^2$ with that in \cite{freiberg}. For $d_{vol}=-\frac{1}{4}f^{1234567}$, the equation
(\ref{varphi})
is satisfied if the coefficient of $f^2.f^2$ is $-2$.

By Kozsul's formula, we calculate the Levi-Civita covariant derivatives of $g_{\varphi}$. The nonzero derivatives are:
$$\nabla_{f_2}{f_5}=f_1, \ \ \nabla_{f_2}{f_7}=\frac{1}{2}f_4, \ \ \nabla_{f_5}{f_7}=\frac{1}{2}f_2, \ \ \nabla_{f_7}{f_2}=\frac{1}{2}f_4,$$
$$\nabla_{f_7}{f_3}=-f_1, \ \ \nabla_{f_7}{f_4}=-f_3, \ \ \nabla_{f_7}{f_5}=-\frac{1}{2}f_2, \ \ \nabla_{f_7}{f_6}=-f_5, \ \ \nabla_{f_7}{f_7}=\frac{1}{2}f_6.$$

From the equation (\ref{ucform}), we obtain the cross product of basis elements. The nonzero cross products are:
$$P_{\varphi}(f_1,f_2)=-\frac{1}{2}f_1, \ P_{\varphi}(f_1,f_5)=-\frac{1}{2}f_3, \ P_{\varphi}(f_1,f_6)=-\frac{1}{2}f_4, \ P_{\varphi}(f_1,f_7)=-\frac{1}{4}f_2,$$
$$ P_{\varphi}(f_2,f_3)=-\frac{1}{2}f_3, \ P_{\varphi}(f_2,f_4)=-\frac{1}{2}f_4, \ P_{\varphi}(f_2,f_5)=\frac{1}{2}f_5, \ P_{\varphi}(f_2,f_6)=\frac{1}{2}f_6,$$
$$ P_{\varphi}(f_2,f_7)=-\frac{1}{2}f_7, \ P_{\varphi}(f_3,f_4)=-f_1, \ P_{\varphi}(f_3,f_6)=\frac{1}{2}f_2, \ P_{\varphi}(f_3,f_7)=-\frac{1}{2}f_5,$$
$$ P_{\varphi}(f_4,f_5)=-\frac{1}{2}f_2, \ P_{\varphi}(f_4,f_7)=-\frac{1}{2}f_6, \ P_{\varphi}(f_5,f_6)=-f_7.$$
Note that $P_{\varphi}(f_i,f_j)=-P_{\varphi}(f_j,f_i)$.

To obtain an a.p.m.s from this cross product, the characteristic vector field $\xi=a_1f_1+\ldots+a_7f_7$ should satisfy $g_{\varphi}(\xi,\xi)=-1$. Then $(\phi,\xi,\eta,g)$, where $\phi(X)=P_{\varphi}(\xi,X)$, $g=-g_{\varphi}$ and $\eta(X)=g(\xi,X)$ is an a.p.m.s on the given Lie algebra. Also $\nabla^{g_{\varphi}}=\nabla^g$ and we denote the covariant derivative by $\nabla$, see \cite{bizimparacontact}.

$g_{\varphi}(\xi,\xi)=-1$ implies
$$2a_1a_7-2a_2^2+4a_3a_6+4a_4a_5=-1.$$

Assume that $\nabla_{\xi}{\xi}=0$. Then
$$0=\nabla_{\xi}{\xi}=(a_2a_5-a_7a_3)f_1-a_4a_7f_3+a_2a_7f_4-a_6a_7f_5+\frac{a_7^2}{2}f_6.$$
Since $f_i$ are linearly independent, we get
$$a_7=0, \ \ \ a_2.a_5=0.$$

Let us choose $\xi=\frac{1}{\sqrt{2}}f_2$, that is, $a_2=\frac{1}{\sqrt{2}}$ and $a_i=0$ otherwise. We determine the class the a.p.m.s. $(\phi,\xi,\eta,g)$ belongs to.

Since $\nabla_{\xi}{\xi}=0$, we have $F^{12}=0$ by Theorem \ref{thm1}.

Next we calculate $(\nabla_{\xi}\varphi)(\xi,Y,Z)=g_{\varphi}((\nabla_{\xi}P_{\varphi})(\xi,Y),Z)=0$. For $\xi=\frac{1}{\sqrt{2}}f_2$ and $y=y_1f_1+\ldots+y_7f_7$, we have
$$(\nabla_{\xi}P_{\varphi})(Y)=\nabla_{\xi}(P_{\varphi}(\xi,Y))-P_{\varphi}(\xi,\nabla_{\xi}Y)
=\frac{1}{4}y_5f_1-\frac{1}{8}y_7f_4-\frac{1}{4}y_5f_1+\frac{1}{8}y_7f_4=0$$
and thus $F^{11}=0$ by Theorem \ref{thm2}.

Let $X_0=f_7$. Then
$$\nabla_{\phi(X_0)}\xi=\frac{1}{2}\nabla_{P_{\varphi}(f_2,f_7)}f_2=\frac{1}{2}\nabla_{-\frac{1}{2}f_7}f_2=-\frac{1}{8}f_4\neq 0,$$
which implies that $F^{W_2}\neq 0$ by Theorem \ref{thm5}.

Now we write $F^{10}$. By direct calculation,
$$\eta(Y)=\frac{2}{\sqrt{2}}y_2, \  \nabla_X\xi=\frac{\sqrt{2}}{4}x_7f_4, \ P_{\varphi}(\nabla_X\xi,\xi)=\frac{1}{8}x_7f_4, \  \varphi(\nabla_X\xi,Z,\xi)=\frac{1}{4}x_7z_5,$$

$$\phi(X)=P_{\varphi}(\xi,X)=\frac{1}{2\sqrt{2}}x_1f_1-\frac{1}{2\sqrt{2}}x_3f_3-\frac{1}{2\sqrt{2}}x_4f_4+\frac{1}{2\sqrt{2}}x_5f_5+\frac{1}{2\sqrt{2}}x_6f_6-\frac{1}{2\sqrt{2}}x_7f_7,$$
and
$$\nabla_{\phi(X)}\xi=-\frac{1}{8}x_7f_4, \ -g(\nabla_{\phi(X)}\xi,Z)=g_{\varphi}(\nabla_{\phi(X)}\xi,Z)=\frac{1}{4}x_7z_5.$$

Thus

$
\begin{array}{rcl}
4F^{10}(X,Y,Z)&=&-\eta(Y)\{-\varphi(\nabla_X\xi,Z,\xi)-g(\nabla_{\phi(X)}\xi,Z)-\varphi(\nabla_Z\xi,X,\xi)\\
&&-g(\nabla_{\phi(Z)}\xi,X)\}\\
&&+\eta(Z)\{-\varphi(\nabla_X\xi,Y,\xi)-g(\nabla_{\phi(X)}\xi,Y)-\varphi(\nabla_Y\xi,X,\xi)\\
&&-g(\nabla_{\phi(Y)}\xi,X)\}\\
&=&-\frac{2}{\sqrt{2}}y_2\{-\frac{1}{4}x_7z_5+\frac{1}{4}x_7z_5-\frac{1}{4}x_5z_7+\frac{1}{4}x_5z_7\}\\
&&+\frac{2}{\sqrt{2}}z_2\{-\frac{1}{4}x_7y_5+\frac{1}{4}x_7y_5-\frac{1}{4}x_5y_7+\frac{1}{4}x_5y_7\}\\
&=&0.
\end{array}
$

Similarly, $F^9=0$. Since

$
\begin{array}{rcl}
4(F^5(X,Y,Z)+F^8(X,Y,Z))&=&-\eta(Y)\{F(\phi^2X, \phi^2Z,\xi)-F(\phi X,\phi Z,\xi)\\
&&+F(\phi^2Z, \phi^2X,\xi)-F(\phi Z,\phi X,\xi)\}\\
&&+\eta(Z)\{F(\phi^2X, \phi^2Y,\xi)-F(\phi X,\phi Y,\xi)\\
&&+F(\phi^2Y, \phi^2X,\xi)-F(\phi Y,\phi X,\xi)\}
\end{array}
$
and
$$F(\phi^2X, \phi^2Z,\xi)-F(\phi X,\phi Z,\xi)=-\varphi(\nabla_X\xi,Z,\xi)-g_{\varphi}(\nabla_{\phi X}\xi, Z),$$
for $X=f_7$, $Y=f_2$, $Z=f_5$, we have
$$4(F^5(X,Y,Z)+F^8(X,Y,Z))=\frac{1}{\sqrt{2}}\neq 0.$$
Thus the a.p.m.s contains summands from $\mathbb{G}_5\oplus\mathbb{G}_8$.

Similarly, if $X=f_7$, $Y=f_2$, $Z=f_5$, then
$$4(F^6(X,Y,Z)+F^7(X,Y,Z))=-\frac{1}{\sqrt{2}}\neq 0,$$
implying $F^6+F^7\neq 0$.

In addition,
$$F^3(f_7,f_4,f_6)=\frac{1}{3}\{-\frac{1}{2\sqrt{2}}-\frac{1}{8^3.2\sqrt{2}}-\frac{1}{16\sqrt{2}}-\frac{1}{16.8\sqrt{2}}\}\neq 0,$$
$$F^4(f_7,f_3,f_7)=-\frac{1}{2\sqrt{2}}-\frac{1}{8^3.2\sqrt{2}}-\frac{1}{16\sqrt{2}}-\frac{1}{16.8\sqrt{2}}\neq 0$$
and
$$F^1(f_7,f_7,f_3)+F^2(f_7,f_7,f_3)=\frac{1}{2\sqrt{2}}+\frac{1}{8^3.2\sqrt{2}}-\frac{1}{16\sqrt{2}}-\frac{1}{16.4\sqrt{2}}\neq 0.$$

To sum up, $(\phi,\xi,\eta,g)$, where $\xi=\frac{1}{\sqrt{2}}f_2$ is in
$$\mathbb{G}_1\oplus\mathbb{G}_2\oplus\mathbb{G}_3\oplus\mathbb{G}_4\oplus\mathbb{G}_5\oplus\mathbb{G}_6\oplus\mathbb{G}_7\oplus\mathbb{G}_8.$$

Now we show that there is no normal structure induced by (\ref{3form}). Assume that $(\phi,g,\xi,\eta)$ is a normal structure on $\mathcal{L}$ with $\xi=a_1f_1+\ldots+a_7f_7$. Then since $g_{\varphi}(\xi,\xi)=-1$ and
$\nabla_{\xi}{\xi}=0$ by Corollary \ref{normallik},  we have
$$a_7=0, \ \ a_2.a_5=0, \ \ -2a_2^2+4a_3.a_6+4a_4.a_5=-1.$$

We check (\ref{normalitycondition}) for basis elements $f_1,\ldots,f_7$. For $X=Y=f_2$, the normality condition gives
$$\frac{a_5}{2}(\frac{a_2}{2}f_1+\frac{a_5}{2}f_3+\frac{a_6}{2}f_4)=0,$$
and from linear independence of basis elements, $a_5=0$.

For $X=f_5$ and $Y=f_1$, we get $a_6=0$. Setting $X=f_5$ and $Y=f_7$, implies $a_2=0$, which is a contradiction since
the equation $-2a_2^2+4a_3.a_6+4a_4.a_5=-1$, yields $a_2^2=\frac{1}{2}$ for $a_5=a_6=0$. Also, there is no para-Sasakian (normal and paracontact metric) structure obtained from $\varphi$.

In addition, we do not have any paracontact structure induced by (\ref{3form}). To show this, let $(\phi,g,\xi,\eta)$ be paracontact. Then by Theorem \ref{paracontact}, $\nabla_{\xi}{\xi}=0$. Thus
$$a_7=0, \ \ a_2.a_5=0, \ \ -2a_2^2+4a_3.a_6+4a_4.a_5=-1.$$
Since the a.p.m.s is paracontact,
$$-g_{\varphi}(\phi X,Y)=\frac{1}{2}\{-g_{\varphi}(\nabla_X\xi,Y)+g_{\varphi}(\nabla_Y\xi,X)\}$$
for all $X, Y$. Now $X=f_1$, $Y=f_5$ gives $a_6=0$. For $X=f_1$, $Y=f_6$, we get $a_5=0$ and $X=f_1$, $Y=f_7$ yields a contradiction to $g_{\varphi}(\xi,\xi)=-1$.

Assume that for an a.p.m.s $F^{W_2}=0$. Then for all $X$, we have $\nabla_{\phi X}\xi=0$ by Theorem \ref{thm5}.

$\nabla_{\phi f_1}\xi=0$ implies $a_7=0$. Since $\nabla_{\phi f_2}\xi=0$, we have $a_5=0$. From $\nabla_{\phi f_4}\xi=0$, we get $a_6=0$ and $\nabla_{\phi f_7}\xi=0$ gives $a_2=0$ for $\xi=a_1f_1+\ldots+a_7f_7$. In this case, $g_{\varphi}(\xi,\xi)=0\neq -1$. So $F^{W_2}\neq 0$ for any a.p.m.s induced by (\ref{3form}) on $\mathcal{L}$.

\end{document}